\documentclass{amsart}
\usepackage{amsmath, amssymb, verbatim, amscd, amsthm, bm}
\usepackage[normalem]{ulem}
\usepackage[driverfallback=dvipdfm]{hyperref}
\usepackage[dvipdfm]{graphicx}
\usepackage{color,graphicx,shortvrb}
\usepackage{enumerate}

\newtheorem*{THM}{Theorem}
\newtheorem{lem}{Lemma}[section]

\theoremstyle{definition}

\newtheorem*{ack}{Acknowledgement} 
\numberwithin{equation}{section}

\newcommand{\q}{\quad} 
\newcommand{\qq}{\qquad}

\newcommand{\C}{\mathbb{C}}
\newcommand{\R}{\mathbb{R}}
\newcommand{\Q}{\mathbb{Q}}

\newcommand{\N}{\mathbb{N}}

\newcommand{\CRX}{C_\R(X)}
\newcommand{\CRY}{C_\R(Y)}
\newcommand{\CX}{C(X)}
\newcommand{\CY}{C(Y)}
\newcommand{\ov}{\overline}

\newcommand{\unitx}{\bm{1}_X}
\newcommand{\unity}{\bm{1}_Y}
\newcommand{\V}[1]{\Vert#1\Vert}
\newcommand{\X}{X}
\newcommand{\Y}{Y}
\newcommand{\yz}{y_0}
\newcommand{\zerox}{\bm{0}_\X}
\newcommand{\zeroy}{\bm{0}_\Y}
\newcommand{\TO}{S}
\newcommand{\TOI}{\TO^{-1}}

 \usepackage{marginnote}
 \usepackage{geometry}

\begin{document}

\title[Ring isomorphisms in norm]
{Ring isomorphisms in norm between
Banach algebras of continuous complex-valued functions}

\keywords{Banach--Stone theorem,
Gelfand--Kolmogoroff theorem, ring isomorphism}
\subjclass[2020]{46E15, 46E25, 46J10}

\author[T. Miura]{Takeshi Miura}
\address[T. Miura]
{Department of Mathematics,
Faculty of Science,
Niigata University,
Niigata 950-2181, Japan}
\email{miura@math.sc.niigata-u.ac.jp}

\author[T. Takahashi]{Taira Takahashi}
\address[T. Takahashi]
{Department of Mathematics,
Faculty of Science,
Niigata University,
Niigata 950-2181, Japan}
\email{f24a041e@mail.cc.niigata-u.ac.jp}

\begin{abstract}
Let $\X$ and $\Y$ be compact Hausdorff spaces,
and let $\CX$ and $\CY$ denote the commutative
Banach algebras of all continuous complex-valued
functions on $\X$ and $\Y$, respectively.
We study bijective maps $T$ from $\CX$ onto $\CY$
which preserve the ring structure in the norm
in the following sense:
\[
\V{T(f+g)}=\V{T(f)+T(g)},\q
\V{T(fg)}=\V{T(f)T(g)}
\qq(f,g\in\CX).
\]
Our main objective is to clarify whether such maps
must necessarily be induced by homeomorphisms
between the underlying spaces.
Under the additional assumption
that $T(\ov{f})=\ov{T(f)}$ for $f\in\CX$,
we prove that $T$ is a real-linear isometry.
As a consequence, we obtain a concrete representation
of such maps as weighted composition operators.
\end{abstract}

\maketitle

\section{Introduction and main result}\label{sect1}

Let $\X$ and $\Y$ be compact Hausdorff spaces.
We denote by $\CRX$ and $\CRY$ the commutative
Banach algebras of all continuous real-valued
functions on $\X$ and $\Y$, respectively.
The classical theorem by Banach \cite{bana}
and Stone \cite{ston} states that
every surjective, not necessarily linear,
isometry $T\colon\CRX\to\CRY$ with $T(0)=0$
is a weighted composition operator
induced by a homeomorphism
between $\X$ and $\Y$.
More precisely, the theorem shows that the metric structure
of $\CRX$ completely determines the underlying topological
structure of $\X$.

Gelfand and Kolmogoroff \cite{gelf} obtained
a reverse result.
Namely, algebraic structure of $\CRX$
determines the metric one of it.
More precisely, let $T\colon\CRX\to\CRY$ be
a bijective map that satisfies the following equalities
for all $f,g\in\CRX$:
\[
T(f+g)=T(f)+T(g),\qq
T(fg)=T(f)T(g).
\]
In other words, $T$ is a ring isomorphism
between $\CRX$ and $\CRY$.
Then there exist a homeomorphism
$\tau\colon\Y\to\X$ such that
$T(f)=f\circ\tau$ for all $f\in\CRX$.
Therefore, $T$ preserves the metric structure
of $\CRX$ and $\CRY$.

Dong, Lin and Zheng \cite{dong} generalized
the Gelfand--Kolmogoroff theorem for
\textit{ring isomorphisms in norm}
$T\colon\CRX\to\CRY$
in the following sense:
\[
\V{T(f+g)}=\V{T(f)+T(g)},\q
\V{T(fg)}=\V{T(f)T(g)}
\qq(f,g\in\CRX),
\]
where $\V{\cdot}$ denotes the supremum
norm.

It is natural to consider ring isomorphisms
in norm between $\CX$ and $\CY$,
the commutative Banach algebras
of all continuous complex-valued functions
on $\X$ and $\Y$, respectively.
On the other hand, it is well-known that
even ring isomorphisms on the complex
number field $\C$ behave quite wild
(see. e.g. \cite{kest}):
For example, Charnow \cite{char} proved
that the cardinal number of the set
of all discontinuous ring automorphisms
on $\C$ is $2^{\mathfrak c}$,
where $\mathfrak c$ denotes the cardinality
of the continuum.
Thus, we need an additional assumption
for ring isomorphisms in norm to characterize
them.
The main result of this paper states that
ring isomorphisms in norm $T$ between
$\CX$ and $\CY$ is a real-linear weighted
composition operator induced by
a homeomorphism between $\X$ and $\Y$
if $T$ preserves the complex conjugate.
The following is the main result of this paper.

\begin{THM}
Let $T \colon\CX\to\CY$ be a bijective map
that satisfies the following equalities
for all $f,g\in\CX$:
\begin{align}
\V{T(f+g)}
&=
\V{T(f)+T(g)},
\label{i1}\\
\V{T(fg)}
&=
\V{T(f)T(g)},
\label{i2}\\
T(\ov{f})
&=
\ov{T(f)},
\label{i3}
\end{align}
where $\ov{g}$ denotes
the complex conjugate of $g$.
Then there exist a $u\in\CRY$ with
$u(\Y)\subset\{\pm1\}$,
a homeomorphism $\tau\colon\Y\to\X$
and a closed and open subset $K$
in $\Y$ such that 
the following equality holds
for all $f\in\CX$ and $y\in\Y$:
\[
T(f)(y)=
\begin{cases}
u(y)f(\tau(y))&y\in K,\\[1mm]
u(y)\ov{f(\tau(y))}&y\in\Y\setminus K.
\end{cases}
\]
\end{THM}

Our proof consists of some steps.
The main task is to show that
ring isomorphism in norm is a
real-linear isometry.
This step is essential, since it allows us to reduce the problem
to the classical setting of real-linear isometries.
Once we obtain the property,
we can apply the characterization
for surjective real-linear isometries
to conclude that the map is
represented by a weighted composition
operator.

The paper is organized as follows.
In Section~\ref{sect2}, we establish several auxiliary lemmas
concerning ring isomorphisms in norm.
These results are then combined to prove the main theorem
in Section~\ref{sect2}.

\section{Proof of main theorem}\label{sect2}

In this section, we prove the main theorem stated in Section~\ref{sect1}.
The proof is divided into several steps, clarifying
the metric properties of the given map.

Let $G$ be a group and $B$ a Banach space.
Tabor \cite[Corollary~1]{tabo} proved that
if $U\colon G\to B$
is a surjective map that satisfies
$\V{U(f+g)}=\V{U(f)+U(g)}$
for all $f,g\in G$,
then $U$ is additive.
Applying this result,
we have that a ring isomorphism in norm
$T\colon\CX\to\CY$ satisfies
the following:
\[
T(f+g)=T(f)+T(g)
\qq(f,g\in\CX).
\]
Let $\Q$ be the set of all 
rational numbers.
By induction, one verifies that
\[
T(rf)=rT(f)
\qq(r\in\Q,\,f\in\CX).
\]
We use these two properties frequently
in the rest of the paper.

In the rest of this paper, we denote by
$\unitx$ and $\unity$ the constant functions
on $\X$ and $\Y$, respectively, that take
only the value $1$.

As a first step, we prove that the mapping $T$ maps
real-valued functions on $\X$ to those on $\Y$.

\begin{lem}\label{lem2.1}
The mapping $T$ satisfies
$T(\CRX)\subset\CRY$.
\end{lem}

\begin{proof}
Fix an arbitrary $f\in\CRX$.
We have $T(f)=T(\ov{f})=\ov{T(f)}$
by \eqref{i3}.
This proves that $T(f)\in\CRY$.
Therefore we conclude that
$T(\CRX)\subset\CRY$.
\end{proof}

We next investigate further structural properties of $T$
derived from its additivity.
The next lemma states that $T$ maps
the constant function $\unitx$ to a unimodular
function on $\Y$.
This simplifies the rest of our argument.

\begin{lem}\label{lem2.2}
The range $T(\unitx)(\Y)$ of $T(\unitx)$ is
contained in the two point set $\{\pm1\}$.
\end{lem}

\begin{proof}
First, we prove that $\V{T(\unitx)}=1$.
We derive from \eqref{i2} that
\[
\V{T(\unitx)}=\V{T((\unitx)^2)}
=\V{T(\unitx)^2}
=\V{T(\unitx)}^2.
\]
Hence $\V{T(\unitx)}=0$ or $\V{T(\unitx)}=1$.
Because $T$ is additive, we observe that
$T$ maps the zero function $\zerox$
on $\X$ to $\zeroy$.
Therefore, we conclude that $\V{T(\unitx)}=1$
since $T$ is injective.

Next, we prove that
$|T(\unitx)(y)|=1$ for all $y\in\Y$.
Suppose, on the contrary, that
there exists a $\yz\in\Y$ such that
$|T(\unitx)(\yz)|\neq1$.
Then we have $|T(\unitx)(\yz)|<1$.
By the continuity of the function
$T(\unitx)$, there exists an open
neighborhood $O$ of $\yz$ in $\Y$
such that $|T(\unitx)|<1$ on $O$.
Applying the Urysohn's lemma,
there exists a $u_0\in\CRY$ such that
$u_0(\Y)\subset[0,1]$, $u_0(\yz)=1$
and $u_0=0$ on $\Y\setminus O$.
Because $T$ is bijective,
there exists a $g_0\in\CX$ such that
$T(g_0)=u_0$.
It follows from equality \eqref{i2} that
\[
1=\V{u_0}
=\V{T(g_0)}
=\V{T(g_0)T(\unitx)}
=\V{u_0T(\unitx)}.
\]
On the other hand,
we observe that
$|u_0(y)T(\unitx)(y)|<1$
by the choice of $O$ and $u_0$.
Consequently, we obtain
$\V{u_0T(\unitx)}<1$,
which is in contradiction
with the last equalities.
Therefore, we conclude that
$|T(\unitx)(y)|=1$ for all $y\in\Y$.
By Lemma~\ref{lem2.1},
we have $T(\unitx)\in\CRX$,
which shows that
$T(\unitx)(Y)\subset\{\pm1\}$.
\end{proof}

The next step is to normalize the map $T$
by eliminating $T(\unitx)$.
We define a map $\TO\colon\CX\to\CY$
as follows:
\[
\TO(f)=T(\unitx)T(f)
\qq(f\in\CX).
\]
By Lemma~\ref{lem2.2},
one verifies that $\TO$ is a bijective,
additive map
that satisfies equality
\eqref{i2} for all $f,g\in\CX$.
Moreover, we observe that
$\TO(\unitx)=\unity$,
$\TO(rf)=r\TO(f)$
for $r\in\Q$ and $f\in\CX$
and $\TO(\ov{f})=\ov{\TO(f)}$
for all $f\in\CX$.
This normalization allows us to focus on a unital map,
which simplifies the subsequent arguments.

We next prove that $\TO$ preserves
the positive functions.

\begin{lem}\label{lem2.3}
For each $f\in C(X)$ that satisfies $f\geq0$ on $X$,
we have $\TO(f)\geq0$ on $Y$.
\end{lem}

\begin{proof}
Let $f\in C(X)$ be such that $f\geq0$ on $X$.
Then the square root $\sqrt{f}$ of $f$
is a well-defined continuous real-valued
function on $X$.
There exists a $k\in\N$ such that $\|\TO(\sqrt{f})\|\leq k$.
We set $g=\sqrt{f}/k$ for simplicity in notation.
Because $\TO$ is additive,
we obtain $\V{\TO(g)}\leq1$.
We derive from \eqref{i1}
with $\TO(\unitx)=\unity$ that
\[
\V{\unity-\TO(g^2)}
=\V{\TO(\unitx)-\TO(g^2)}
=\V{\TO(\unitx-g^2)}.
\]
By applying equality \eqref{i2},
we obtain
\[
\V{\TO(\unitx-g^2)}
=\V{\TO(\unitx-g)\TO(\unitx+g)}
=\V{(\unity-\TO(g))(\unity+\TO(g))}
=\V{\unity-\TO(g)^2},
\]
where we have used that $\TO$ is a unital additive map.
Combining the previous equalities shows that
\[
\V{\unity-\TO(g^2)}
=\V{\unity-\TO(g)^2}.
\]
By Lemma~\ref{lem2.2},
$\TO(g)$ and $\TO(g^2)$ are both real-valued functions on $Y$
with $\V{\TO(g)}\leq1$.
It follows that $0\leq \TO(g)^2(y)\leq1$ for all $y\in Y$,
thus $\V{\unity-\TO(g)^2}\leq1$.
We derive from the last equality that
\[
|1-\TO(g^2)(y)|
\leq
\V{\unity-\TO(g^2)}
=\V{\unity-\TO(g)^2}\leq1
\]
for all $y\in Y$.
Since $\TO(g^2)$ is a real-valued function,
we conclude that $\TO(g^2)\geq0$ on $Y$.
Therefore, we conclude that
$\TO(f)=\TO(k^2g^2)
=k^2\TO(g^2)\geq0$ on $Y$.
\end{proof}

The next lemma shows that $\TO$ is bounded.
This will play an important role to prove that
$\TO$ is a real-linear isometry.

\begin{lem}\label{lem2.4}
The inequality $\V{\TO(f)}\leq\V{f}$
holds for all $f\in\CX$.
\end{lem}

\begin{proof}
Fix an arbitrary $f\in\CX$.
First, we prove that
$\V{\TO(|f|)}=\V{\TO(f)}$.
Because $\V{\cdot}$ is the supremum norm,
we obtain $\V{\TO(|f|)}^2=\V{\TO(|f|)^2}$.
Applying equality \eqref{i2} twice,
we have the following equalities:
\[
\V{\TO(|f|)^2}
=\V{\TO(|f|^2)}
=\V{\TO(f\ov{f})}
=\V{\TO(f)\TO(\ov{f})}.
\]
By equality \eqref{i3}, we have
$\TO(\ov{f})=\ov{\TO(f)}$,
which implies that
\[
\V{\TO(f)\TO}(\ov{f})
=\V{\TO(f)\ov{\TO(f)}}
=\V{|\TO(f)|^2}
=\V{\TO(f)}^2.
\]
Combining the previous equalities gives
$\V{\TO(|f|)}^2=\V{\TO(f)}^2$,
consequently $\V{\TO(|f|)}=\V{\TO(f)}$.

Let $r\in\Q$ be arbitrary with
$\V{f}\leq r$.
Then we have $r\unitx-|f|\geq0$ on $\X$.
Applying Lemma~\ref{lem2.3} shows that
$\TO(r\unitx-|f|)\geq0$ on $\Y$.
Because $\TO$ is additive,
we obtain the following inequalities:
\[
0\leq
\TO(r\unitx-|f|)
=\TO(r\unitx)-\TO(|f|)
=r\unity-\TO(|f|),
\]
where we have used that
$\TO(r\unitx)=r\TO(\unitx)$
and $\TO(\unitx)=\unity$.
Combining the previous equalities
with $\V{\TO(|f|)}=\V{\TO(f)}$,
we obtain the following inequalities:
\[
\V{\TO(f)}
=\V{\TO(|f|)}
\leq\V{r\unity}=r,
\]
hence $\V{\TO(f)}\leq r$.
Because $r\in\Q$ with
$\V{f}\leq r$ was chosen arbitrarily,
we conclude that $\V{\TO(f)}\leq\V{f}$.
\end{proof}

\begin{proof}[\textbf{Proof of main theorem}]
Let $\TO\colon\CX\to\CY$ be as above.
By Lemma~\ref{lem2.4}, we have that
$\TO$ is a bounded additive map.
One verifies that $\TO$ is a real-linear
map.
Therefore, we conclude that
$\TO$ is a bijective, bounded
real-linear map between Banach spaces
$\CX$ and $\CY$.
By the open mapping theorem,
we obtain that the inverse map
$\TOI\colon\CY\to\CX$ of $\TO$ is
bounded as well.
Then there exists an $M>0$ such that
$\V{\TOI(u)}\leq M\V{u}$ holds
for all $u\in\CY$.
It follows that $\V{f}\leq M\V{\TO(f)}$
for all $f\in\CX$.
Fix $f\in\CX$ arbitrarily.
We prove that $\V{f}\leq\V{\TO(f)}$.
Combining the previous inequality
with \eqref{i2},
we have the following inequalities:
\[
\V{f}^2
=\V{f^2}
\leq
M\V{\TO(f^2)}
=M\V{\TO(f)^2}
=M\V{\TO(f)}^2.
\]
By induction, we observe that
$\V{f}^{2^n}\leq M\V{\TO(f)}^{2^n}$
for all $n\in\N$.
Consequently, we obtain
$\V{f}\leq M^{1/2^n}\V{\TO(f)}$
for all $n\in\N$.
Letting $n\to\infty$ gives
$\V{f}\leq\V{\TO(f)}$ as desired.
By combining the last inequality
with Lemma~\ref{lem2.4},
we conclude that $\TO$ is
a norm preserving map.
Since $\TO$ is real-linear,
we have that $\TO$ is a surjective
real-linear isometry.
Now we can apply a result by 
Ellis \cite[Theorem]{elli},
which characterized surjective
real-linear isometries between
function algebras.
Then there exist a homeomorphism
$\tau\colon\Y\to\X$
and a closed and open subset $K$
in $\Y$ such that the following
identity holds for all $f\in\CX$
and $y\in\Y$:
\[
\TO(f)(y)=
\begin{cases}
f(\tau(y))&y\in K,\\[1mm]
\ov{f(\tau(y))}&y\in\Y\setminus K.
\end{cases}
\]
By definition, we have
$\TO(f)=T(\unitx)T(f)$
for each $f\in\CX$.
We set $u=T(\unitx)$.
Then we obtain $u\in\CRY$
with $u(Y)\subset\{\pm1\}$.
We conclude from the above
form of $\TO$ that
the following equality holds
for all $f\in\CX$ and $y\in\Y$:
\[
T(f)(y)=
\begin{cases}
u(y)f(\tau(y))&y\in K,\\[1mm]
u(y)\ov{f(\tau(y))}&y\in\Y\setminus K.
\end{cases}
\]
This completes the proof.
\end{proof}

\begin{ack}
The first author was partially supported by
JSPS KAKENHI Grant Number JP 25K07028.
\end{ack}

\end{document}